\documentclass[aps,floatfix,groupedaddress,superscriptaddress,longbibliography,footinbib,notitlepage,showpacs]{revtex4-1}
\usepackage{graphicx,graphics,times,bm,bbm,bbold,amssymb,soul}
\usepackage[T1]{fontenc}
\usepackage{amsthm,amsmath,amsfonts,dsfont,color,xcolor,dcolumn,mathrsfs,mathtools}
\usepackage{enumitem}
\usepackage{nicematrix,tikz,mathdots}
\definecolor{Bittersweet}{HTML}{C04F17}
\usepackage[bookmarksnumbered,colorlinks,plainpages,linktocpage]{hyperref}
\definecolor{Darkgreen}{rgb}{0,0.4,0}
\hypersetup{%
    pdfborder = {0 0 0},
    colorlinks,
    citecolor=blue,
    filecolor=Darkgreen,
    linkcolor=Bittersweet,
    urlcolor=blue}
\usepackage{tocloft}

\def\Hess{{\rm Hess}}

\usepackage{tabularx,ragged2e}

\makeatletter
\renewcommand{\@fnsymbol}[1]{\ifcase#1\or a\or b\or c \else\@ctrerr\fi}
\makeatother

\newtheorem{theorem}{Theorem}
\newtheorem{remark}[theorem]{Remark}
\newtheorem{lemma}[theorem]{Lemma}
\newtheorem{proposition}[theorem]{Proposition}
\newtheorem{corollary}[theorem]{Corollary}

\newtheorem{definition}[theorem]{Definition}
\newtheorem{example}[theorem]{Example}

\makeatletter
\newcommand{\mytableofcontents}{%
  \begingroup
    \let\clearpage\relax  
    \section*{Contents}    
    \@starttoc{toc}        
  \endgroup
}
\makeatother

\begin{document}

\title{Symmetric Persistent Tensors and their Hessian}

\author{Masoud Gharahi} 
\email[]{masoud.gharahi@gmail.com}
\affiliation{Faculty of Physics, Astronomy and Applied Computer Science, Jagiellonian University, 30-348 Kraków, Poland}
\author{Giorgio Ottaviani}
\email[]{giorgio.ottaviani@unifi.it}
\affiliation{Department of Mathematics and Computer Science Ulisse Dini, University of Florence, 50134 Florence, Italy}

\begin{abstract}
Persistent tensors, introduced in [\href{https://doi.org/10.22331/q-2024-01-31-1238}{Quantum 8 (2024), 1238}], and inspired by quantum information theory, form a recursively defined class of tensors that remain stable under the substitution method and thereby yield nontrivial lower bounds on tensor rank. In this work, we investigate the symmetric case—namely, symmetric persistent tensors, or equivalently, persistent polynomials. We establish that a symmetric tensor in $\mathrm{Sym}^n \mathbbm{C}^d$ is persistent if the determinant of its Hessian equals the $d(n-2)$-th power of a nonzero linear form. The converse is verified for cubic tensors ($n=3$) or for $d \leq 3$, by leveraging classical results of B. Segre. Moreover, we demonstrate that the Hessian of a symmetric persistent tensor factors as the $d$-th power of a form of degree $(n-2)$. Our main results provide an explicit necessary and sufficient criterion for persistence, thereby offering an effective algebraic characterization of this class of tensors. Beyond characterization, we present normal forms in small dimensions, place persistent polynomials within prehomogeneous geometry, and connect them with semi-invariants, homaloidal polynomials, and Legendre transforms. Particularly, we prove that all persistent cubics are homaloidal.
\end{abstract}

\maketitle



\mytableofcontents
\section{Introduction and Statements of the Main Results}\label{sec.i}

Persistent tensors were introduced in \cite{GL24} to formalize a notion of stability under contraction. In the simplest case of two factors, persistence coincides with conciseness, meaning that no factor is redundant in the tensor. For tensors with three or more factors, the definition becomes recursive: a tensor is persistent if it is concise and if, after contracting along any direction outside a linear subspace in one of the factors, the resulting smaller tensor remains persistent. In this sense, persistence ensures that the essential structure of a tensor survives successive contractions, never degenerating into a trivial form.

In this work, we focus on the symmetric case, where tensors correspond to homogeneous polynomials. Symmetric persistent tensors form a natural and distinguished subclass with particularly stable structural behavior. In degree two, persistence recovers the classical notion of a nonsingular quadratic form, i.e., one of maximal rank. For higher degrees, persistent polynomials are those that are both concise (involving no redundant variables) and whose directional derivatives are themselves persistent for almost every direction. Equivalently, persistence can be viewed as the property that successive differentiation preserves the full structural richness of the polynomial.

From a classical algebraic viewpoint, non-concise polynomials correspond to cones, while persistent polynomials lie at the opposite extreme: their derivatives maintain maximal independence and avoid collapsing into lower-dimensional degeneracies. In this way, persistent polynomials provide a natural generalization of quadratic forms of maximal rank (nonsingular quadratic forms) to higher degrees. They define a canonical family of homogeneous polynomials with maximal stability under differentiation, offering a bridge between modern concepts from quantum information theory and classical results in algebraic geometry. In fact, as shown in \cite{GL24}, families of minimal-rank persistent tensors can be regarded as natural generalizations of multiqubit $\mathsf{W}$ states \cite{DVC} to the multiqudit setting, retaining an analogous entanglement robustness under particle loss: taking derivatives with respect to any variable preserves linear independence, mirroring the way $\mathsf{W}$ states remain entangled when any qubit is traced out. Equivalently, they lie in the orbit closure of the generalized $\mathsf{GHZ}$ state, a canonical entanglement structure of central importance in quantum information science.

The key phenomenon under investigation is the strong link between persistence and the Hessian. We establish that a symmetric tensor is persistent if its Hessian admits a highly constrained factorization, namely as a power of a linear form. This connection yields both an effective algebraic criterion for persistence and, in special cases, a complete classification. Beyond the characterization theorem, the study extends to small-dimensional instances, recovering classical results of B.~Segre and providing explicit normal forms. Further, the work places persistent polynomials within the broader context of prehomogeneous vector spaces, revealing connections to semi-invariants, homaloidal polynomials, and Legendre transforms.

We may identify, as usual, symmetric tensors with homogeneous polynomials (see Sec. \ref{subsec:polrest} for details). 
We state the recursive definition of persistent in the setting of homogeneous polynomials. For further details see Sec. \ref{subsec:persistent}.
\begin{definition}\label{def:persistent-symtensor_intro}
Let $V$ be a finite dimensional vector space.
\begin{itemize}
\item[(i)] A homogeneous polynomial $f\in \mathrm{Sym}^2V$ is persistent if it is a quadratic form of maximal rank, or nonsingular, or concise.
\item[(ii)] A homogeneous polynomial $f\in \mathrm{Sym}^nV$ with $n > 2$ is persistent if it is concise and there exists a hyperplane $S \subsetneq V^\vee$ such that the first derivative $\frac{\partial f}{\partial u}\in \mathrm{Sym}^{n-1}V$ is persistent whenever $u\notin S$. 
\end{itemize}
\end{definition}
The main point of this paper is a strong connection between persistent homogeneous polynomials and their Hessian. The Hessian $\Hess(f)$ of a homogeneous polynomial $f(x_0,\ldots, x_{d-1})$ of degree $n$ in $d$ variables is the determinant of its Hessian matrix (see Sec. \ref{subsec:hessian} for details). It is a homogeneous polynomial of degree $d(n-2)$ in the same $d$ variables.  

We recall that it is still an open problem to describe the polynomials $f$ such that $\Hess(f)$ vanishes identically \cite{Ru}. The locus of such polynomials is $\mathrm{SL}(d)$-invariant, and it is closed.

The locus of polynomials \(f\) for which the Hessian satisfies $\Hess(f)=\lambda \ell^{d(n-2)}$, for some scalar $\lambda\in\mathbbm{C}$ and linear form $\ell$, is again 
$\mathrm{SL}(d)$-invariant. It is closed upon allowing $\lambda=0$, which corresponds to the previously considered case.

Our main result is the following chain of implications.

\begin{theorem}\label{thm:main}
Let $f(x_0,\ldots, x_{d-1})\in\mathrm{Sym}^n\mathbbm{C}^d$. Consider the following conditions:
\begin{enumerate}[label=(\alph*), font=\normalfont]
\item There exists a nonzero linear homogeneous polynomial $\ell$ such that
\begin{equation}
\Hess(f)=\ell^{d(n-2)}\,.
\end{equation}
\item $f$ is persistent.
\item For $(n-2)$ sets of coordinates $u^{(1)},\ldots,u^{(n-2)} \in \mathbbm{C}^d$, one has
\begin{equation}\label{eq:hessx}
\Hess_x\Bigg(\sum_{i_1=0}^{d-1}\cdots\sum_{i_{n-2}=0}^{d-1} u^{(1)}_{i_1}\cdots u^{(n-2)}_{i_{n-2}}
\frac{\partial^{n-2}f(x)}{\partial x_{i_1}\cdots\partial {x_{i_{n-2}}}}\Bigg) = \left[g\big(u^{(1)},\ldots, u^{(n-2)}\big)\right]^d\,,
\end{equation}
where the Hessian is taken with respect to $x=(x_0,\ldots,x_{d-1})$, and $g$ is a nonzero multi-homogeneous polynomial of multidegree $(1,\ldots,1)$  in the sets of variables $u^{(1)},\ldots,u^{(n-2)}$.
\item There exists a nonzero homogeneous polynomial $g$ of degree $(n-2)$ such that
\begin{equation}
\Hess(f)=g^{d}\,.
\end{equation}
\end{enumerate}
Then the following implications hold:
\[
(\textnormal{a}) \;\Longrightarrow\; (\textnormal{b})
\;\Longleftrightarrow\; (\textnormal{c}) \;\Longrightarrow\; (\textnormal{d})\,.
\]
\end{theorem}

In the left-hand side (LHS) of (\ref{eq:hessx}), the polynomial $\frac{\partial^{n-2}f(x)}{\partial x_{i_1}\cdots\partial x_{i_{n-2}}}$ is of degree $2$; hence, the Hessian does not depend on $x$. Note that the equivalence between $(\textnormal{b})$ and $(\textnormal{c})$ provides an effective algorithm to detect whether a polynomial is persistent. At the same time, $(\textnormal{a})$ gives a sufficient condition, and $(\textnormal{d})$ gives a necessary condition. While $(\textnormal{d})$ is certainly not sufficient (see Example \ref{exa:hessbinary} (i)), it is currently unknown whether $(\textnormal{a})$ is a necessary condition. This holds for $n=3$ (see Corollary \ref{coro:n=3}) and for $d\leq 4$ (see Theorem \ref{thm:small_dim}).

Some examples are useful to illustrate the statement of Theorem \ref{thm:main}.

\begin{example}\label{exa:hessbinary}
We illustrate the statement of Theorem~\ref{thm:main} with three explicit examples. In each case, persistence could be verified directly ``by hand,'' but our aim is to demonstrate 
how Theorem \ref{thm:main} can be applied.
\begin{itemize}
\item[(i)]
Consider $h=x_0^4+x_1^4\in\mathrm{Sym}^4\mathbbm{C}^2$. For notational convenience, we set $u^{(1)} = u$ and $u^{(2)} = v$. 
The LHS of~\eqref{eq:hessx} becomes
\[
\Hess_x\big(12(u_0v_0x_0^2+u_1v_1x_1^2)\big)=12^2\det\begin{pmatrix}2u_0v_0&0\\0&2u_1v_1\end{pmatrix}=24^2u_0u_1v_0v_1\,.
\]
Since this expression is not a square, $h$ is not persistent, 
in view of the equivalence $(\textnormal{b}) \Longleftrightarrow (\textnormal{c})$ in Theorem~\ref{thm:main}. However, the restitution (substituting $u_0=v_0=x_0$ and $u_1=v_1=x_1$) yields $24^2(x_0x_1)^2=4\,\Hess(h)$, which is now a square.

\item[(ii)] Next, consider $g=x_0^2x_1^2\in\mathrm{Sym}^4\mathbbm{C}^2$. In this case, the LHS of~\eqref{eq:hessx} is
\begin{align*}
\Hess_x\big(2u_0v_0x_1^2+4(u_0v_1+u_1v_0)x_0x_1+2u_1v_1x_0^2\big)&=\det\begin{pmatrix}4u_1v_1&4(u_0v_1+u_1v_0)\\
4(u_0v_1+u_1v_0)&4u_0v_0\end{pmatrix} \\
&=4^2\big(u_0u_1v_0v_1-(u_0v_1+u_1v_0)^2\big) \\
&=4^2\big(-u_0^2v_1^2-u_1^2v_0^2-u_0u_1v_0v_1\big).
\end{align*}
As in case~(i), the expression is not a square, and therefore $g$ is not persistent. The restitution gives $-48(x_0x_1)^2=4\,\Hess(g)$, which is a scalar multiple of $\Hess(h)$. Note, however, that the expressions in $(u,v)$ arising from $h$ and $g$ in~\eqref{eq:hessx} differ; they ``remind'' indeed the determinantal representations 
from which the Hessian $(x_0 x_1)^2$ was obtained. At first sight, this appears to contradict the bijectivity of polarization and restitution (see Section~\ref{subsec:polrest}). The resolution lies in the fact that our Hessians are quartic polynomials: a complete polarization would require four variables $u^{(1)}, \ldots, u^{(4)}$, whereas we have restricted to just $u^{(1)}$ and $u^{(2)}$. Thus, the expression on the LHS of~\eqref{eq:hessx} 
should be interpreted as a partial polarization. This partiality provides sufficient flexibility to express persistence, and explains why condition~$(\textnormal{c})$ in Theorem~\ref{thm:main} remains valid.

\item[(iii)] Finally, consider the persistent example $f=x_0^3x_1\in\mathrm{Sym}^4\mathbbm{C}^2$. In this case, the LHS of (\ref{eq:hessx}) is
\[
\Hess_x\big(6u_0v_0x_0x_1+3(u_0v_1+u_1v_0)x_0^2\big)=\det\begin{pmatrix}3(u_0v_1+u_1v_0)&6u_0v_0\\
6u_0v_0&0\end{pmatrix}=-36(u_0v_0)^2\,,
\]
which is indeed a square. Consequently, $f$ satisfies condition~$(\textnormal{c})$ in Theorem~\ref{thm:main}, 
which certifies its persistence. Moreover, the restitution yields $-36x_0^4=4\,\Hess(f)$, which even satisfies condition~$(\textnormal{a})$ in Theorem~\ref{thm:main}.
\end{itemize}
Less trivial examples are in Example \ref{exa:SolT}.
\end{example}

\begin{corollary}\label{coro:n=3}
Let $f\in\text{Sym}^3\mathbbm{C}^d$.
The following conditions are equivalent
\begin{itemize}
\item[(a)] There exists a nonzero linear homogeneous polynomial $\ell$ such that $\Hess(f)=\ell^{d}$.
\item
[(b)] $f$ is persistent
\end{itemize}
\end{corollary}
\begin{proof}
For $n=3$, conditions $(\textnormal{a})$ and $(\textnormal{d})$ in Theorem \ref{thm:main} are equivalent.
\end{proof}

In order to state our next result, we recall that the weight of a monomial
${\bf{x}}^{\alpha}=x_0^{\alpha_0}\cdots x_{d-1}^{\alpha_{d-1}}$ is $\sum i\alpha_i$, which amounts to declare that the variable $x_i$ has weight $i$.
The space generated by monomials of weight $j$ are the invariant subspaces in $\text{Sym}^n\mathbbm{C}^d$ for the action of the torus $\mathbbm{C}^*$ acting as $t\cdot x_i=t^ix_i$. Recall that all actions of $\mathbbm{C}^*$ are conjugated to a diagonal one where $t\cdot x_i=t^{\omega_i}x_i$ for some weight $\omega_i\in{\mathbbm Z}$, extending it as $t\cdot {\bf{x}}^{\alpha}=x_0^{\omega_0\alpha_0}\cdots x_{d-1}^{\omega_{d-1}\alpha_{d-1}} $.

We now specialize our general characterization to small dimensions, where explicit normal forms can be given. 
The following theorem provides a complete classification of persistent symmetric tensors in dimensions $d\leq 3$ and $(d=4, n=3)$, showing that in these cases persistence is precisely equivalent to the Hessian being a power of a nonzero linear form.

\begin{theorem}\label{thm:small_dim}
\textnormal{(I)} A symmetric tensor $f\in\text{Sym}^n\mathbbm{C}^d$ for $d\leq 3$ is a persistent tensor iff its Hessian is the $d(n-2)$-th power of a nonzero linear form $\ell$.
\begin{itemize}[leftmargin=1.6cm]
\item[($d=2$).] $f=x_0^{n-1}x_1$ up to $\mathrm{GL}(2,\mathbbm{C})$ (i.e., $f$ is an isobaric of weight $1$) $\Leftrightarrow$ $\Hess(f)=\ell^{2(n-2)}$ $\Leftrightarrow$ $f$ is persistent.

\item[($d=3$).] $f=x_0^{n-1}x_2+x_0^{n-2}x_1^2$ up to $\mathrm{GL}(3,\mathbbm{C})$ (i.e., $f$ is an isobaric of weight $2$)  $\Leftrightarrow$ $\Hess(f)=\ell^{3(n-2)}$ $\Leftrightarrow$ $f$ is persistent.
\end{itemize}

\textnormal{(II)} A symmetric tensor $f\in\text{Sym}^3\mathbbm{C}^4$ is a persistent tensor iff its Hessian is the fourth power of a nonzero linear form $\ell$:
\begin{align*}
f=\lambda_1x_0^{2}x_3+\lambda_2x_0x_1x_2+\lambda_3x_1^3 \text{~up~to~}\mathrm{GL}(4,\mathbbm{C})\text{~for~some~scalar~}\lambda_i~\text{with~}\lambda_1\lambda_2\neq 0\text{~(i.e.,~}f~\text{is~an~isobaric~of~weight~3)~}\\
\Leftrightarrow \Hess(f)=\ell^4 \Leftrightarrow f~\text{is~persistent}.
\end{align*}
\end{theorem}

\begin{remark}
For $d=2, 3$ there is only one $\mathrm{SL}(d)$-orbit. For $(d=4, n=3)$, there are exactly two $\mathrm{SL}(4)$-orbits, which are respectively the Chasles-Cayley ruled surface (for $\lambda_3\neq 0)$ and the union of a smooth quadric surface with a tangent plane (for $\lambda_3=0$). The latter orbit  is in the closure of the former.

In connection with the entanglement classification (see \cite{GMO}), the normal forms identified in Theorem~\ref{thm:small_dim} acquire a natural interpretation in quantum information theory. Specifically:
\begin{itemize}
\item for $d=2$, the normal form is identical to the $n$-qubit $\mathsf{W}$ state of \cite{DVC};
\item for $d=3$, the normal form coincides with the so-called $n$-qutrit $\mathsf{Y}$ state defined in \cite{GM};
\item for $(d=4, n=3)$, if $\lambda_{3}\neq 0$ and $\lambda_{3}=0$, the normal form yields the so-called $3$-ququart $\mathsf{L}$ and $\mathsf{M}$ states, respectively, both introduced in \cite{GL24}.
\end{itemize}
Thus, the algebraic–geometric classification of orbits underlying persistent tensors aligns precisely with the well-established hierarchy of entanglement families.
\end{remark}

The remainder of the paper is organized as follows. Section \ref{sec.ii} reviews the necessary preliminaries, covering polarization and restitution, the Hessian and the polar map, persistent tensors, and isobaric polynomials. Section \ref{sec.iii} establishes our main characterization theorem, showing how persistence is intrinsically linked to the factorization of the Hessian. Section \ref{sec.iv} turns to the low-dimensional setting, where explicit classifications of persistent tensors can be given. Finally, in Section \ref{sec.v} we prove that cubic persistent polynomials are homaloidal, following an argument by Mammana. Moreover Section \ref{sec.v} situates persistent polynomials within the framework of prehomogeneous spaces, presenting concrete examples, revisiting classical constructions, and highlighting connections with the multiplicative Legendre transform.

\section{Preliminaries}\label{sec.ii}

This section sets the groundwork for the rest of the paper by clarifying definitions, tools, and prior results that are essential to the study of persistent polynomials.

\subsection{Polarization and Restitution}\label{subsec:polrest}

We consider the natural bijective correspondence between symmetric tensors with $d$ modes and homogeneous polynomials of degree $d$. In other words $\mathrm{Sym}^n\mathbbm{C}^d$ and $\mathbbm{C}[x_0,\ldots, x_{d-1}]_n$ are naturally isomorphic.  This correspondence can be easily understood in terms of the classical operators of \emph{polarization} and \emph{restitution}.

\begin{definition}
Let $P\in\mathbbm{C}[x_0,\ldots, x_{d-1}]_n$ be a homogeneous $d$-variate degree-$n$ polynomial, and set $V=\mathbbm{C}^d$.  
The \emph{polarization} of $P$ is the symmetric multilinear map $pol(P)\colon V^{n}\to\mathbbm{C}$ defined by
\begin{equation}
pol(P)(u^{(1)},\ldots, u^{(n)}) \coloneqq \frac{1}{n!}\sum_{i_1=0}^{d-1}\cdots\sum_{i_n=0}^{d-1}
u^{(1)}_{i_1}\cdots u^{(n)}_{i_n}\frac{\partial^n P}{\partial x_{i_1}\cdots \partial x_{i_n}}\,.
\end{equation}
\end{definition}

Conversely, given a multilinear symmetric map $T\colon V^{n}\to\mathbbm{C}$, its \emph{restitution} is defined by 
$T(x,\ldots, x)$. The restitution is precisely the inverse operation of polarization. In particular, the following fundamental relation holds:
\begin{equation}
P(x)=pol(P)(x,\ldots, x).
\end{equation}
This yields the identity
\begin{equation}\label{eq:polfull}
P(x) = \frac{1}{n!}\sum_{i_1=0}^{d-1}\cdots\sum_{i_n=0}^{d-1}
x_{i_1}\cdots x_{i_n}\frac{\partial^n P}{\partial x_{i_1}\cdots \partial x_{i_n}}\,,
\end{equation}
which is obtained by iterating $n$ times the Euler relation
\begin{equation}
P(x)=\frac{1}{n}\sum_{i=0}^{d-1}x_i\frac{\partial P}{\partial x_i}\,.
\end{equation}

This procedure can be further generalized to the case of \emph{partial polarization}, which give natural expressions for the partial derivatives of $P$. 

The first cases are given by
\begin{align}
\frac{\partial P}{\partial a}(x)&=pol(P)(a,x,\ldots, x), \\
\frac{\partial^2 P}{\partial a^2}(x)&=pol(P)(a,a,x,\ldots, x),
\end{align}
and so on. These identities correspond respectively to the formulas
\begin{align}\label{eq:polpartial1}
\sum_{i_1=0}^{d-1}u_{i_1}\frac{\partial P}{\partial x_{i_1}}(x) &= \frac{1}{(n-1)!}\sum_{i_1=0}^{d-1}\cdots\sum_{i_n=0}^{d-1}
u_{i_1}x_{i_2}\cdots x_{i_n}\frac{\partial^n P}{\partial x_{i_1}\cdots\partial x_{i_n}}\,, \\ \label{eq:polpartial2}
\sum_{i_1=0}^{d-1}\sum_{i_2=0}^{d-1}u_{i_1}u_{i_2}
\frac{\partial^2 P}{\partial x_{i_1}\partial x_{i_2}}(x) &= \frac{1}{(n-2)!}\sum_{i_1=0}^{d-1}\cdots\sum_{i_n=0}^{d-1}
u_{i_1}u_{i_2}x_{i_3}\cdots x_{i_n}\frac{\partial^n P}{\partial x_{i_1}\cdots\partial x_{i_n}}\,,
\end{align}
and so on. Note that (\ref{eq:polpartial1}) is obtained by applying (\ref{eq:polfull}) to $\frac{\partial P}{\partial x_{i_1}}$, while (\ref{eq:polpartial2}) arises in the same way from $\frac{\partial^2 P}{\partial x_{i_1}\partial x_{i_2}}$. Note that the expression appearing in Theorem \ref{thm:main}, in the LHS of (\ref{eq:hessx}), before computing $\Hess_x$, is a partial polarization of $f$. 

\subsection{Persistent Tensors}\label{subsec:persistent}

In order to understand the persistent tensors we need to know conciseness of tensors. Informally, a tensor is concise if it cannot be written as a tensor in a smaller ambient space, which means that the tensor uses all dimensions of the local spaces. Therefore, conciseness can be characterized as all single-party reduced density matrices being full rank.

\begin{definition}\label{def:concise}
A tensor $\mathcal{T}\in V_1\otimes\cdots\otimes V_n$ is called concise in the first factor, or $1$-concise, if $\mathcal{T}\notin V'_1\otimes V_2\otimes\cdots\otimes V_n$ for every proper subspace $V'_1\subsetneq V_1$. Conciseness with respect to the other factors is defined analogously. A tensor is called \emph{concise} if it is $i$-concise for all $i\in [n]$.
\end{definition}

Persistent tensors are defined in the general case in Refs. \cite{Gha,GL24}, we adopt here the same notational conventions.
\begin{definition}\label{def:persistent-tensor}
Let $V_i$ be finite dimensional vector spaces.
\begin{itemize}
\item[(i)] A tensor $\mathcal{P}\in V_1\otimes V_2$ is persistent if it is $1$-concise.
\item[(ii)] A tensor $\mathcal{P}\in V_1\otimes\cdots\otimes V_n$ with $n > 2$ is persistent if it is $1$-concise and there exists a subspace $S \subsetneq V_1^{\vee}$ such that the contraction $\langle u|\mathcal{P}\in V_2\otimes\cdots\otimes V_n$ is persistent whenever $\langle u|\notin S$.
\end{itemize}
\end{definition}

Recall symmetric tensors in $\mathrm{Sym}^nV$ correspond to homogeneous polynomials of degree $n$. A polynomial $P\in\mathrm{Sym}^nV$ is concise if its partial derivatives $\partial_i P$ are linearly independent. In the classical language, non-concise polynomials are the so called cones.

In the symmetric case, we get the following special case.

\begin{definition}\label{def:persistent-symtensor}
Let $V$ be a finite dimensional vector space.
\begin{itemize}
\item[(i)] A symmetric tensor $\mathcal{F}\in \mathrm{Sym}^2V$ is persistent if it is a quadratic form of maximal rank, or nonsingular, or concise.
\item[(ii)] A symmetric tensor $\mathcal{F}\in \mathrm{Sym}^nV$ with $n > 2$ is persistent if it is concise and there exists a subspace $S \subsetneq V^\vee$ such that the contraction $\langle u|\mathcal{F}\in \mathrm{Sym}^{n-1}V$ is persistent whenever $\langle u|\notin S$. We can assume $\mathrm{codim}~S=1$ without loss of generality.
\end{itemize}
\end{definition}

As a consequence of the correspondence recalled in Sec. \ref{subsec:polrest}, we can refer equivalently to {\it symmetric persistent tensors} or to {\it persistent polynomials}.

\subsection{The Hessian and the Polar Map}\label{subsec:hessian}
Here, we define the Hessian matrix and polynomial, recall its $\mathrm{GL}(d)$-equivariance, and introduce the polar map and homaloidal polynomials.

\begin{definition}
Let $f\in\mathbbm{C}[x_0,\ldots,x_{d-1}]$ be a polynomial.  The \emph{Hessian matrix} of $f$ is the symmetric $d\times d$ matrix
\begin{equation}
\mathcal{H}_f \coloneqq \left( \frac{\partial^2 f}{\partial x_i \, \partial x_j} \right)_{0 \leq i,j \leq d-1}.
\end{equation}
The \emph{Hessian polynomial} of $f$ is the determinant of its Hessian matrix, that is,
\begin{equation}\label{Hessian}
\Hess(f)\coloneqq\det(\mathcal{H}_f)\in\mathbbm{C}[x_0,\ldots,x_{d-1}]\,.
\end{equation}
We shall refer to $\Hess(f)$ simply as the \emph{Hessian} of $f$.
\end{definition}

We denote by $V(f)\subset\mathbbm{P}^{d-1}$ the projective hypersurface defined by a homogeneous polynomial $f\in\mathbbm{C}[x_0,\ldots,x_{d-1}]$.
A projective hypersurface has vanishing Hessian if and only if the derivatives $\{\partial_i f\}_{i=0}^{d-1}$ are algebraically dependent, which means there is a non-zero polynomial $g\in\mathbbm{C}[x_0,\ldots,x_{d-1}]$ such that $g(\partial_0 f,\ldots,\partial_{d-1}f)=0$.

Gordan and Noether proved that projective hypersurfaces with vanishing Hessian are cones when $d \leq 4$ \cite{GN}. In this range, they correspond to non-concise tensors, which will be recalled in Definition~\ref{def:concise}. For $d\geq 5$, however, there exist concise symmetric tensors associated with polynomials having vanishing Hessian (see Example~\ref{exa:weight4}). For a general reference on these topics, see~\cite{Ru}.

The natural action of $\mathrm{GL}(d)$ on $\mathbbm{C}^d$ extends to the polynomial ring $\mathbbm{C}[x_0,\dots,x_{d-1}]$ via
\begin{equation}
(g\cdot f)(x) = f(g^{-1}x), \quad \forall\, g\in \mathrm{GL}(d)\,.
\end{equation}
The Hessian is $\mathrm{GL}(d)$-equivariant in the sense that, for any homogeneous polynomial $f\in\mathbbm{C}[x_0,\ldots,x_{d-1}]$, there exists an integer $N\in\mathbbm{Z}$ such that
\begin{equation}
\Hess(g \cdot f)=\det(g)^N\,\Hess(f), \quad \forall\, g \in \mathrm{GL}(d)\,.
\end{equation}
A recent reference discussing group actions and the Hessian is \cite{COCDR}.

\begin{definition}
The gradient of a homogeneous polynomial $f$, $\nabla f([x_0:\dots:x_{d-1}]) 
= [ \frac{\partial f}{\partial x_0} : \cdots : \frac{\partial f}{\partial x_{d-1}}]$, induces a rational map 
\begin{equation}\label{polar-map}
\nabla f \colon \mathbbm{P}^{d-1} \dashrightarrow (\mathbbm{P}^{d-1})^\vee,
\end{equation}
called the \emph{polar map}. Its restriction to the hypersurface 
$V(f) \subset \mathbbm{P}^{d-1}$ is the classical Gauss map. The image of the Gauss map is the dual variety $V(f)^\vee \subset (\mathbbm{P}^{d-1})^\vee$.
\end{definition}

\begin{definition}
A symmetric tensor $f$ is called \emph{homaloidal} if the polar map $\nabla f$ in \eqref{polar-map} is a birational map.
\end{definition}

Polynomials with vanishing Hessian have degree of the polar map equal to zero. Homaloidal polynomials make the next case, when the degree of the polar map is equal to one.

\subsection{Isobaric Polynomials and their Hessian}

We now define isobaric polynomials, determine how their weights influence the form of the Hessian, and present key lemmas illustrating persistence obstructions.

\begin{definition}
A homogeneous polynomial $f\in\text{Sym}^n\mathbbm{C}^d$  is called \emph{isobaric} of weight $\delta$ if all the monomials ${\bf{x}}^{\alpha}=x_0^{\alpha_0}\cdots x_{d-1}^{\alpha_{d-1}}$ in its support satisfy
\begin{equation}
\sum_{j=0}^{d-1}j\alpha_j=\delta\,.
\end{equation}
\end{definition}

\begin{lemma}\label{lem:weight}
Let $f\in\text{Sym}^n\mathbbm{C}^d$ be an isobaric polynomial of weight $\delta$. Then $\Hess(f)\in\text{Sym}^{d(n-2)}\mathbbm{C}^d$ and is isobaric of weight $d(\delta-d+1)$. In particular:
\begin{itemize}
\item[(i)] $\Hess(f)=0$ if $\delta<d-1$.
\item[(ii)] $\Hess(f)=\lambda x_0^{d(n-2)}$ for some $\lambda\in\mathbbm{C}$ if $\delta=d-1$. This scalar $\lambda$ may vanish, but it does not vanish for a general $f\in\text{Sym}^n\mathbbm{C}^d$ isobaric of weight $\delta$.
\end{itemize}
\end{lemma}
\begin{proof}
The $(i,j)$-entry of the Hessian matrix, for $i,j=0,\ldots, d-1$,
has degree $d-2$ and is isobaric of weight $\delta-i-j$.
Hence every term in the determinant of the Hessian matrix has weight $d\delta-\sum_{i=0}^{d-1}i-\sum_{j=0}^{d-1}j=
d(\delta-d+1)$, as we claimed. First item follows immediately: if $\delta<d-1$, then the weight of $\Hess(f)$ becomes negative, so $\Hess(f) = 0$. Second item follows similarly: if $\delta=d-1$, then the weight of $\Hess(f)$ is zero. In this case, the Hessian is nonzero for a general $f$, which is evident from the triangular shape of the Hessian matrix; see Example~\ref{exa:weight4} for illustration.
\end{proof}

\begin{example}\label{exa:weight4} For $d=5$ and $n=3$, any isobaric $f\in\rm{Sym}^3\mathbbm{C}^5$ of weight $4$ has the following form
\[
f=\lambda_1x_0^2x_4+\lambda_2x_0x_1x_3+\lambda_3x_0x_2^2+\lambda_4x_1^2x_2\,.
\]
We get the Hessian matrix is
\[
\begin{pmatrix}
*&*&*&*&2\lambda_1x_0\\
*&*&*&\lambda_2x_0&0\\
*&*&2\lambda_3x_0&0&0\\
*&\lambda_2x_0&0&0&0\\
2\lambda_1x_0&0&0&0&0
\end{pmatrix},
\]
so that
\[
\Hess(f)=8\lambda_1^2\lambda_2^2\lambda_3x_0^5\,,
\]
which is nonzero for general $\lambda_i$. For $\lambda_1=0$ the polynomial is not concise ($x_4$ is missing) and
$\Hess(f)=0$. For $\lambda_3=0$ the polynomial is concise but still $\Hess(f)=0$, this is a famous example constructed by Perazzo \cite{Per}, of a concise polynomial with vanishing Hessian, that we review in Proposition \ref{prop:P4prehom} (see Ref. \cite{Ru} as a general reference on these topics). Here, $f$ is persistent if $\lambda_1\lambda_2\lambda_3\neq0$.
\end{example}

\section{Characterization of Symmetric Persistent Tensors}\label{sec.iii}

This section develops and proves Theorem \ref{thm:main}. The equivalences between conditions $(\textnormal{a})$–$(\textnormal{d})$ are established, providing both effective criteria and an algorithm for testing whether a tensor is persistent.

To facilitate the proof of the main result, we first establish the following lemma.

\begin{lemma}\label{lem:multd}
Let $n \geq 4$. For each $i=1,\ldots,n-2$, let $u^{(i)}=(u^{(i)}_0,\ldots,u^{(i)}_r)$ denote homogeneous coordinates on the $i$-th factor of the product $\mathbbm{P}^r\times\cdots\times\mathbbm{P}^r$ ($(n-2)$ times). Let $G(u^{(1)},\ldots,u^{(n-2)})$ be a multihomogeneous symmetric polynomial of multidegree $(d,\ldots,d)$. Assume that for general choices $\tilde u^{(2)},\ldots,\tilde u^{(n-2)}$, there exists a linear form $\ell_{\tilde u^{(2)},\ldots,\tilde u^{(n-2)}}(u)$ such that $G(u,\tilde u^{(2)},\ldots,\tilde u^{(n-2)})
=(\ell_{\tilde u^{(2)},\ldots,\tilde u^{(n-2)}}(u))^d$. Then there exists a multilinear form $m(u,u^{(2)},\ldots,u^{(n-2)})$ such that $G(u,u^{(2)},\ldots,u^{(n-2)})
= \big(m(u,u^{(2)},\ldots,u^{(n-2)})\big)^d$.
\end{lemma}
\begin{proof}
Decompose $G$ into irreducible factors. Pick an irreducible factor $G_1$ of multidegree $(e_1,\ldots, e_{n-2})$. By the symmetry there are analogous factors for any permutation  of the $e_i$, hence we may assume $e_1\geq\cdots\geq e_{n-2}$ and consider the projections from the variety $V(G_1)\subset\mathbbm{P}^r\times\cdots\times\mathbbm{P}^r$ to the last $(n-3)$ factors $\mathbbm{P}^r\times\cdots\times\mathbbm{P}^r$.
Since $e_1\ge 1$, then the projection is a dominant map between irreducible varieties. Comparing degrees, we get $G_1(u,\tilde u^{(2)},\ldots,\tilde u^{(n-2)})=(\ell_{\tilde u^{(2)},\ldots,\tilde u^{(n-2)}}(u))^{e_1}$, this means that the general schematic fiber (which is an hypersurface in $\mathbbm{P}^r$) is nonreduced at every point, this contradicts \cite[Lemma III.10.5]{Hart}. If $e_1=1$ then the multidegree of $G_1$ is $(1,\ldots, 1,0,\ldots, 0)$. Assume $e_{k}=1$ and $e_{k+1}=0$. If $k=n-2$ then we have a balanced $G_1$ of bidegree $(1,\ldots, 1)$ such that  $G_1(u,\tilde u^{(2)},\ldots ,\tilde u^{(n-2)})=(\ell_{\tilde u^{(2)},\ldots, \tilde u^{(n-2)}}(u))$
for general  $\tilde u^{(2)},\ldots ,\tilde u^{(n-2)}$, then for any $\tilde u^{(2)},\ldots, \tilde u^{(n-2)}$.  
By the assumption we get $G(u,u^{(2)},\ldots, u^{(n-2)})=\big(G_1(u,u^{(2)},\ldots, u^{(n-2)}\big)^d$ as we wanted.
If $k=1$ then $\prod_{i=1}^{n-2} h(u^{(i)})$ divides $G$
for a linear form $h$, hence $\ell_{\tilde u^{(2)},\ldots, \tilde u^{(n-2)}}(u)$ does not depend on ${\tilde u^{(2)},\ldots, \tilde u^{(n-2)}}$. By the assumption we get $G=\prod_{i=1}^{n-2} \big(h(u^{(i)})\big)^d$ which gives the thesis. If $1<k<n-2$, then $n\geq 5$ and we have that $\prod_{i=1}^{k+1}h(u^{(1)},\ldots,\widehat{u^{(i)}},\ldots,u^{(k+1)})$ divides $G$. By the assumption $h(u^{(1)},\ldots,\widehat{u^{(i)}},\ldots, u^{(k+1)})$ is the same linear form for any $ u^{(2)},\ldots, u^{(k+1)}$ and this implies that $h(u^{(1)},\ldots,\widehat{u^{(i)}},\ldots, u^{(k+1)})$ does not depend on $u^{(2)},\ldots,u^{(k+1)}$ which contradicts $e_2=1$ (it should be $e_2=0$), in other words only the extreme cases $k=1$ and $k=n-2$ are possible. This concludes the proof.
\end{proof}

{\bf Proof of the equivalence $(\textnormal{b})$ $\Longleftrightarrow$ $(\textnormal{c})$.}

The LHS of (\ref{eq:hessx}) is a multihomogeneous polynomial $G(u^{(1)},\ldots, u^{(n-2)})$ of multidegree $(d,\ldots, d)$.

We now explain the proof of the implication $(\textnormal{b}) \Longrightarrow (\textnormal{c})$. Let us first consider the case $n=3$. By assumption, there is a linear form $\ell$ such that $f_v$ is a nonsingular quadric whenever $v \notin \{\ell=0\}$. Consequently, $\sum_{i=0}^{d-1}u^{(1)}_i\frac{\partial f}{\partial x_i}$ is nonsingular quadratic for $u^{(1)} \notin \{\ell=0\}$. This implies that
\begin{equation}
\Hess_x\Bigg(\sum_{i=0}^{d-1}u^{(1)}_i\frac{\partial f}{\partial x_i}\Bigg)=\Big(\ell(u^{(1)})\Big)^d,
\end{equation}
where we have rescaled $\ell$. Note that $\ell$ is nonzero, since for general $u^{(1)}$ it is nonzero by the assumption of persistency. Equivalently, any pencil of quadrics contains at least a singular member. This establishes $(\textnormal{c})$ in the case $n=3$.

We next illustrate the case $n=4$. By assumption, there is a linear form $\ell$ such that $\sum_{i=0}^{d-1}u^{(1)}_i\frac{\partial f}{\partial x_i}$ is a persistent cubic for $u^{(1)}\notin \{\ell=0\}$. This entails that
\begin{equation}
\sum_{i=0}^{d-1} \sum_{j=0}^{d-1}u^{(1)}_i u^{(2)}_j\frac{\partial^2 f}{\partial x_i\partial x_j}
\end{equation}
is a nonsingular quadric whenever $u^{(1)}\notin \{\ell=0\}$ and $u^{(2)}\notin \{\ell_{u^{(1)}}=0\}$, where $\ell_{u^{(1)}}$ denotes a linear form depending on $u^{(1)}$. Hence
\begin{equation}
\Hess_x\Bigg(\sum_{i=0}^{d-1}
\sum_{j=0}^{d-1}u^{(1)}_iu^{(2)}_j\frac{\partial^2 f}{\partial x_i\partial x_j}\Bigg) = G\big(u^{(1)},u^{(2)}\big),
\end{equation}
where $G$ is multihomogeneous of multidegree $(d,d)$, symmetric in $(u^{(1)},u^{(2)})$, and satisfies the property that for general $u^{(1)}$, there exists at most one hyperplane $H$ in the coordinate $u^{(2)}$ such that $G(u^{(1)},u^{(2)})\neq 0$ for $u^{(2)}\notin H$. This implies that for general $u^{(1)}$, $G(u^{(1)},u^{(2)})$ is a $d$-power of a linear form in $u^{(2)}$. It follows from Lemma \ref{lem:multd} that $G=g^d$, where $g$ is multihomogeneous of multidegree $(1,1)$.

For the general case, the assumption of persistence implies that for generic $(u^{(2)}, \ldots, u^{(n-2)})$, the polynomial $G$ vanishes on at most one hyperplane in the coordinate $u^{(1)}$. Moreover, this behaviour is symmetric with respect to permutations of the variables $u^{(i)}$. Hence we deduce that $G=g^d$, again from Lemma \ref{lem:multd}, where $g$ is multihomogeneous of multidegree $(1,\ldots,1)$.

This completes the proof of the implication $(\textnormal{b}) \Longrightarrow (\textnormal{c})$, and, moreover, provides the foundation for establishing the reverse implication as well.

When \eqref{eq:hessx} holds, for general $(u^{(2)}, \ldots, u^{(n-2)})$, the quadric
\begin{equation}
\sum_{i_1=0}^{d-1}\cdots\sum_{i_{n-2}=0}^{d-1} u^{(1)}_{i_1}\cdots u^{(n-2)}_{i_{n-2}}
\frac{\partial^{n-2}f(x)}{\partial x_{i_1}\cdots\partial {x_{i_{n-2}}}}
\end{equation}
is singular for at most one hyperplane in the coordinate $u^{(1)}$. This property, in turn, establishes the persistence of $f$.

The implication $(\textnormal{a}) \Longrightarrow (\textnormal{b})$ in Theorem \ref{thm:main} follows directly from the following proposition, since $(\textnormal{a}^{\prime})$ is evidently stronger than $(\textnormal{c})$, which, in turn, is equivalent to $(\textnormal{b})$.

\begin{proposition}\label{prop:aaprime}
Let $f\in \mathrm{Sym}^n\mathbbm{C}^d$ and $\ell$ be a nonzero linear form.  The following are equivalent
\begin{itemize}
\item[$(\textnormal{a})$] $\left((n-2)!\right)^d\Hess(f)=\ell^{d(n-2)}$.
\item[$(\textnormal{a}^{\prime})$] For $(n-2)$ sets of coordinates $u^{(1)},\ldots, u^{(n-2)}\in\mathbbm{C}^d$, one has
\begin{equation}\label{eq:hessxl}
\Hess_x\Bigg(\sum_{i_1=0}^{d-1}\cdots\sum_{i_{n-2}=0}^{d-1} u^{(1)}_{i_1}\cdots u^{(n-2)}_{i_{n-2}}
\frac{\partial^{n-2}f(x)}{\partial x_{i_1}\cdots\partial {x_{i_{n-2}}}}\Bigg) = \Big[\ell(u^{(1)})\cdots \ell(u^{(n-2)})\Big]^d\,,
\end{equation}
 where the Hessian is taken with respect to $x=(x_0,\ldots, x_{d-1})$.
\end{itemize}
\end{proposition}
\begin{proof}
We begin by observing that
\begin{equation}\label{eq:hessianstep}
\Hess_x\Bigg(\sum_{i_1=0}^{d-1}\cdots\sum_{i_{n-2}=0}^{d-1} u^{(1)}_{i_1}\cdots u^{(n-2)}_{i_{n-2}}
\frac{\partial^{n-2}f(x)}{\partial x_{i_1}\cdots\partial {x_{i_{n-2}}}}\Bigg)=
\det\Bigg(\sum_{i_1=0}^{d-1}\cdots\sum_{i_{n-2}=0}^{d-1} u^{(1)}_{i_1}\cdots u^{(n-2)}_{i_{n-2}}
\frac{\partial^{n}f(x)}{\partial x_{i_1}\cdots\partial {x_{i_{n}}}}\Bigg)_{i_{n-1},i_n}.
\end{equation}
By restitution, setting $u^{(1)}=\cdots=u^{(n-2)}=x$ in \eqref{eq:hessianstep} yields
\begin{equation}\label{eq:restituted}
\det\Bigg(\sum_{i_1=0}^{d-1}\cdots\sum_{i_{n-2}=0}^{d-1} x_{i_1}\cdots x_{i_{n-2}}
\frac{\partial^{n}f(x)}{\partial x_{i_1}\cdots\partial {x_{i_{n}}}}\Bigg)_{i_{n-1},i_n}.
\end{equation}
By virtue of equation~\eqref{eq:polfull}, the expression in \eqref{eq:restituted} simplifies to
\begin{equation}
\det\Bigg((n-2)!
\frac{\partial^{2}f(x)}{\partial x_{i_{n-1}}\partial {x_{i_{n}}}}\Bigg)_{i_{n-1},i_n}=\left((n-2)!\right)^d\Hess(f)\,.
\end{equation}

The implication $(\textnormal{a}^{\prime}) \Longrightarrow (\textnormal{a})$ follows. In order to prove the reverse implication $(\textnormal{a}) \Longrightarrow (\textnormal{a}^{\prime})$, let $G(u^{(1)},\ldots, u^{(n-2)})$ be the LHS of (\ref{eq:hessxl}), it is a multihomogeneous polynomial of multidegree $(d,\ldots, d)$. Moreover, it is symmetric under permutations of $u^{(i)}$ and it satisfies by assumption
\begin{equation}\label{eq:rank1}
G(u,\ldots, u)=\left[\ell(u)\right]^{d(n-2)}.
\end{equation}
Identifying $\mathrm{Sym}^n \mathbbm{C}^d = U$, the polynomial $G$ can be regarded as a symmetric tensor in $U^{\otimes(n-2)}$, in fact lying in $\mathrm{Sym}^{n-2} U$. Equation \eqref{eq:rank1} then shows that $G$ has rank one. This implies there exists $H \in U$, that is, a homogeneous polynomial of degree $(n-2)$ on $\mathbbm{C}^d$, such that $G(u^{(1)},\ldots, u^{(n-2)})=H(u^{(1)})\cdots H(u^{(n-2)})$. Hence we get $G(u,\ldots,u)=[H(u)]^{(n-2)}$, and by \eqref{eq:rank1} we conclude that $H(u) = [\ell(u)]^d$. This completes the proof of $(\textnormal{a}^{\prime})$.
\end{proof}

The implication $(\textnormal{c}) \Longrightarrow (\textnormal{d})$ in Theorem \ref{thm:main} follows by restitution as in the proof of Prop. \ref{prop:aaprime}, after adjusting a scalar multiple to $g$.

\begin{remark}
We do not know whether the implication $(\textnormal{b}) \Longrightarrow (\textnormal{a})$ in Theorem~\ref{thm:main} holds in general. It is valid, however, when $n=3$ or $d \leq 4$, as established by Corollary~\ref{coro:n=3} and Theorem~\ref{thm:small_dim}, respectively. On the other hand, the implication $(\textnormal{d}) \Longrightarrow (\textnormal{c})$ in Theorem~\ref{thm:main} fails, with a counterexample provided in Example~\ref{exa:hessbinary}~(i) or~(ii).

\end{remark}

\begin{remark}\label{rem:bsegre}
B. Segre, in \cite[(20)]{Seg60}, provides the example
\begin{equation}
f(x)=x_0^{2m}+(1-2m)x_1^{2m}-x_2^mx_3^m\,,
\end{equation}
whose Hessian is a perfect square, explicitly given by
\begin{equation}
\Hess(f)=\left[2m^2(2m-1)^2(x_0x_1x_2x_3)^{m-1}\right]^2\,.
\end{equation}

Moreover, he establishes a geometric condition for the Hessian of a quaternary form $f$ to be a square modulo $f$, involving the asymptotic lines of the surface $V(f)$.
\end{remark}
\section{Classification of Symmetric Persistent Tensors in Small Dimension}\label{sec.iv}

In what follows, we present a detailed proof of Theorem \ref{thm:small_dim}, focusing on the cases $d=2,3,4$. Building on classical results and applying inductive techniques, we obtain a complete classification of persistent polynomials in these dimensions, providing explicit normal forms together with a description of the associated geometric loci.

\begin{theorem}
Let $d=2$ and $f$ be persistent. Then $f=x_0^{n-1}x_1$ up to $\mathrm{GL}(2,\mathbbm{C})$.
\end{theorem}
\begin{proof}
For $n=2, 3$ the result is straightforward. In this case, the subspace $S$ is a point in $\mathbbm{P}^1$. We may work by induction on $n\geq 4$. If $\frac{\partial f}{\partial p}$ is persistent for $p\notin S$ then $\frac{\partial f}{\partial p}$ is a form of degree $n-1$ which has, by induction, a root of multiplicity $n-2$. Then in the linear system $\frac{\partial f}{\partial p}$ the root of multiplicity $n-2$ must be constant by Bertini Theorem. Hence, $f$ itself has the same root of multiplicity $n-1$.
\end{proof}

\begin{theorem}
Let $d=3$ and $f$ be persistent. Then $f=x_0^{n-2}(x_0x_2+x_1^2)$ up to $\mathrm{GL}(3,\mathbbm{C})$.
Geometrically, this is a line counted $n-2$ times and a tangent conic.
\end{theorem}
\begin{proof}
In this case, the subspace $S$ is a line in $\mathbbm{P}^2$. For $n=3$ the result is well known (see e.g. Table 2 in Ref. \cite{Ban}), since the general plane cubic can be reduced, up to $\mathrm{GL}(3,\mathbbm{C})$-action, to the Hesse form $x_0^3+x_1^3+x_2^3+3\lambda x_0x_1x_2$. We continue by induction on $n$. 

If $\frac{\partial f}{\partial p}$ is persistent for $p\notin S$ then $\frac{\partial f}{\partial p}$ is a form of degree $n-1$ which has, by induction, a component consisting of a line of multiplicity $n-2$. 
For $n=4$, we get a linear system of cubics, such that its general member is a cubic containing a line. Then the general member is singular in two points that, by Bertini theorem, must be in the base locus. Hence the two points are the same for the general member of the linear system, this implies that $\frac{\partial f}{\partial p}$ contains the same line $\ell$ for general $p$. By Euler relation, $f$ itself contains the line $\ell$. We may assume $\ell=x_0$. Then $f=x_0g$ with $g$ a cubic. Computing derivatives $f_0=g+x_0g_0$, $f_1=x_0g_1$, $f_2=x_0g_2$, since the general linear combination $\alpha_0f_0+\alpha_1f_1+\alpha_2f_2$ is divisible by $x_0$, it follows that $g$ is divisible by $x_0$ and we get $f=x_0^2h$ with $h$ a conic, as we wanted. We show later that the conic $h$ has to be tangent to the line $V(x_0)$. 

For $n\geq 5$ the argument is a bit shorter. By Bertini theorem $\frac{\partial f}{\partial p}$ is a form of degree $n-1$ that contains $\ell^{n-3}$ as a factor for a line $\ell$, hence 
 $\frac{\partial f}{\partial p}$ is singular on a line $\ell$. It follows by Bertini Theorem that the line $\ell$ is in the base locus and does not depend on $p$. By Euler relation, $f$ itself contains the line $\ell$. By induction,  the general linear combination $\alpha_0f_0+\alpha_1f_1+\alpha_2f_2$ is divisible by $x_0^{n-3}$ and again by Euler relation $f=x_0^{n-3}g$ with $g$ a cubic. Computing derivatives $f_0=(n-3)x_0^{n-4}g+x_0^{n-3}g_0$, $f_1=x_0^{n-3}g_1$, $f_2=x_0^{n-3}g_2$, since the general linear combination $\alpha_0f_0+\alpha_1f_1+\alpha_2f_2$ is divisible by $x_0^{n-3}$ it follows that $g$ is divisible by $x_0$ and we get $f=x_0^{n-2}h$ with $h$ a conic, as we wanted. The conic is nonsingular by conciseness. Then by $\mathrm{SL}(3)$-action we have the tangent case, as in the statement (where we may invoke \cite[Prop. 2.7]{COCDR}) or the secant case
$x_0^{n-2}(x_1x_2+x_0^2)$. The corresponding Hessian in the secant case is
$(n-1)x_0^{3n-8}\left((n-2)x_1x_2-nx_0^2\right)$, which is not a cube. Hence, according to (d) of Theorem \ref{thm:main}, the secant case is not persistent. Equivalently, we can check that
$f$ has weight $2$ up to $\mathrm{SL}(3)$-action.
\end{proof}

\begin{corollary}
Let $d=3$, $n\geq 2$. Let $\Hess(f)=\lambda\ell^{3(n-2)}$ with $\ell$ a linear form, $\lambda$ any scalar, possibly zero. Then there are two cases:
\begin{itemize}
\item[(i)] $f$ is a cone, this case consists of a variety of dimension $n+2$ in $\mathbbm{C}[x_0,x_1,x_2]_n$.
\item[(ii)] $f=\ell^{n-2}q$ with $q$ a smooth conic tangent to $\ell$. 

When $q$ become singular, in the closure of (ii), we have the intersection of the two cases, of dimension $5$ (when $n\geq 3$) , while the closure of the second case has dimension $6$ (again when $n\geq 3$).
\end{itemize}
\end{corollary}
\begin{remark}
For $n \geq 4$, the locus $\{f\mid\Hess(f)=\lambda\ell^{3(n-2)}\}$ has two irreducible components. Moreover, for $n=3,4$, the variety $\{\ell^{n-2}q\mid q~\textrm{is a conic tangent to}~\ell\}$ coincides with the locus of forms of highest rank, which is $5$ for $n=3$ and $7$ for $n=4$ (this was already observed by B.~Segre, see \cite[\S 97]{Seg42}). By contrast, for $n=5$, Buczyński and Teitler \cite{BT} show that there exists a symmetric non-persistent polynomial of rank $10$, which exceeds the rank $9$ of $\ell^3 q$. Similarly, for $n=6$, De~Paris \cite{DeP} constructs further non-persistent sextic polynomials with strictly higher rank. Finally, the persistent polynomial $\ell^{n-2}q$ has rank $(n-1)(d-1)+1$, as shown in \cite[Proposition~3.7]{CG24}, in agreement with the nonsymmetric case discussed in \cite{GL24}. Theorem~1 in \cite{GL24} establishes that any persistent tensor in $\mathrm{Sym}^n \mathbbm{C}^d$ has tensor rank at least $(n-1)(d-1)+1$. 
Since the symmetric (Waring) rank is always greater than or equal to the tensor rank, the same lower bound applies to the symmetric rank as well. The first computation of the rank of $\ell q$ was carried out in \cite{CGV}, where certain estimates for the real rank were also obtained. More generally, estimating the real rank of a real persistent polynomial remains an interesting problem.
\end{remark}

\begin{theorem}  
Let $f\in\mathrm{Sym}^n\mathbbm{C}^4$ be persistent. If $n=3$ then $f=\lambda_1x_0^{2}x_3+\lambda_2x_0x_1x_2+\lambda_3x_1^3$ up to $\mathrm{GL}(4,\mathbbm{C})$.

If $n=4$ then $V(f)$ has a triple line. If $\ell=\{x_0=x_1=0\}$ this means that $f\in (x_0, x_1)^{3}$.
\end{theorem}
\begin{proof}
In Ref. \cite[\S 95]{Seg42} it is proved that the only cubic surface with Hessian given by $\ell^4$, with $\ell$ a nonzero linear form, is $\mathrm{GL}(4,\mathbbm{C})$-equivalent to $f$ in the statement, which is called a Chasles-Cayley ruled cubic surface (Ref. \cite[pag. 141]{Seg42}) if $\lambda_3\neq 0$ or it splits in a smooth quadric and a tangent plane if $\lambda_3=0$.
By condition $(\textnormal{d})$ in Theorem \ref{thm:main}, this proves the result for $n=3$. 

Let $n=4$ and $f$ be persistent. The only nonreduced $f$ is given by $\ell^2q$, where $\ell$ is a linear form and $q$ is a quadric tangent to $\ell$ (see \cite[Prop. 2.7]{COCDR} and the fact that $q^2$ is not persistent.). This case has two triple lines, so we may assume $f$ reduced. If $\frac{\partial f}{\partial p}$ is persistent for $p\notin S$ then $\frac{\partial f}{\partial p}$ is a cubic form  which is singular on a line. 
By Bertini Theorem (and the fact that $f$ is reduced) such a line is the same for the general member of the linear system and $f$ itself is singular on this line, that we may assume $V(x_0,x_1)$.

Hence $\frac{\partial f}{\partial v}\in (x_0,x_1)^2$ for $v$ on a dense subset. By continuity we get that $\frac{\partial f}{\partial v}\in (x_0,x_1)^2$ $\forall v$. We claim that $f\in(x_0,x_1)^3$. Since $(x_0,x_1)^3$ is a monomial ideal, for this claim we may assume that $f$ is a monomial, and in this case the claim is obvious. This concludes the proof.
\end{proof}

\begin{remark}
For the Chasles-Cayley ruled surface ($n=3$ in the statement) the singular locus is the line $(x_0=x_1=0)$ with the embedded point $(0,0,0,1)$ so that the isobaric coordinates may be found geometrically from the singular locus.
\end{remark}

\begin{remark}
For $d=4$, in the closure of persistent cubics there are cones over plane cubics which have cusps or two double points.
\end{remark}

\begin{example}\label{exa:SolT}
Example~10~(b) in \cite{SolT} considers the polynomial $f = x_0x_2^3 + x_1x_2^2x_3 + x_3^4$. A direct computation shows that $\Hess(f) = 9x_2^8$. Hence, $f$ is persistent by condition~\textnormal{(a)} of Theorem~\ref{thm:main}. Note that $f$ does not possess any linear factor; therefore, it is not 
$\mathrm{GL}(4,\mathbbm{C})$–equivalent to a weight–$3$ polynomial, which would necessarily contain $x_0$ as a factor.

Example~14 in \cite{SolT} considers the polynomial $g = x_2^4 + x_0x_2^2x_3 + x_1x_2x_3^2 + x_3^4$. In this case, $\Hess(g) = 9(x_2x_3)^4$, so $g$ satisfies condition~\textnormal{(d)} of Theorem~\ref{thm:main}. However, $g$ is not persistent, since
\[
\Hess_x\!\Bigg(\sum_{i,j=0}^3 u_i v_j 
\frac{\partial^2 g}{\partial x_i \partial x_j}\Bigg)
= 16\big(u_3^2 v_2^2 + u_2u_3v_2v_3 + u_2^2v_3^2\big)^2
\]
is not a fourth power and thus fails to satisfy condition~\textnormal{(c)} of Theorem~\ref{thm:main}.
\end{example}
\section{Persistent Polynomials and Prehomogeneous Spaces: Examples and Insights}\label{sec.v}

\subsection{Persistent Cubic Polynomials are Homaloidal }

The results of this subsection are essentially due to Mammana and can be found in \cite[n.6]{Mam}.
We start by recalling two classical results. Although well known to experts, we include straightforward proofs for the sake of completeness.

\begin{proposition}\label{prop:contains_line}
Let $V$ be a linear system of quadrics in $\mathbbm{P}^r$ of affine dimension $r+1$, assumed to be base-point free. Fix a point $M\in\mathbbm{P}^r$ and a line $L\ni M$. Suppose that every quadric in $V$ containing $M$ is tangent to $L$ at $M$; equivalently, the subsystem of quadrics through $M$ has codimension one. Then, among the quadrics tangent to $L$ at $M$, there exists at least one quadric that is singular at $M$.
\end{proposition}

\begin{proof}
After a suitable projective change of coordinates, we may assume $M = (1:0:\cdots:0) \in \mathbbm{P}^r$,
and $L = \{(s:t:0:\cdots:0) \mid (s,t) \in \mathbbm{P}^1\}$.
Let $A=(a_{ij})_{i, j=0,\ldots, r}$ be the symmetric matrix associated with a quadric of the system. The quadric contains $M$ if and only if $a_{00}=0$. The quadric is tangent to $L$ at $M$ if and only if $a_{00}=a_{01}=0$, so, by assumption, if a quadric of the system satisfies $a_{00}=0$ then it satisfies also $a_{01}=0$. The quadric is singular at $M$ if and only if $a_{0i}=0$ for all $i=0,\ldots, r$. Hence the condition of singularity at $M$, among the quadrics containing $M$, requires only $r-1$ linear conditions
$a_{0i}=0$ for all $i=2,\ldots, r$, since $a_{01}=0$ is already granted. This proves the result. 
\end{proof}

\begin{lemma}\label{lem:pencil}
A pencil of quadrics on $\mathbbm{P}^1$ with no base points has two distinct double points.
\end{lemma}

\begin{proof}
Let $(x_0, x_1)$ be homogeneous coordinates on $\mathbbm{P}^1$.
We may assume that the pencil is $\lambda(x_0-ax_1)(x_0-bx_1)+\mu(x_0-cx_1)(x_0-dx_1)$
for some $a, b, c, d\in\mathbbm{C}$ and parameters $(\lambda, \mu)$. Expanding this expression we get $(\lambda+\mu)x_0^2-\left(\lambda(b+a)+\mu(d+c)\right)x_0x_1+(\lambda ab+\mu cd)x_1^2$, which has discriminant 
$\left(\lambda(b+a)+\mu(d+c)\right)^2-4(\lambda+\mu)(\lambda ab+\mu cd)=\lambda^2((b+a)^2-4ab)+\lambda\mu(2(b+a)(d+c)-4ab-4cd)+\mu^2((d+c)^2-4cd)$. The two roots in $(\lambda,\mu)$ give the two pencils with double points. The last equation in turn has discriminant $(2(b+a)(d+c)-4ab-4cd)^2-4((b+a)^2-4ab)((d+c)^2-4cd)=16(b-d)(b-c)(a-d)(a-c)$, which vanishes if and only if the pencil has a base point.
\end{proof}

\begin{theorem}\cite[n.6]{Mam}
Let $f$ be a cubic polynomial such that $\Hess(f)$ is the power of a nonzero linear form. Then $f$ is homaloidal. In particular, any persistent cubic polynomial is homaloidal.
\end{theorem}
\begin{proof}
In the case of two variables the statement can be easily checked directly, hence we may assume there are at least three variables.
Consider the vector space spanned by the partials $V=\langle f_0,\ldots, f_{d-1}\rangle$,  by the assumption that the Hessian is nonzero we get $\dim V=d$.
This space defines a rational map $\nabla f\colon\mathbbm{P}^{d-1}\to\mathbbm{P}^{d-1}$, which is dominant since its Jacobian coincides with the Hessian of $f$.
The linear system $V$ of quadrics has base points exactly on the singular locus $Sing\,V(f)$.Assume that for general $P\in\mathbbm{P}^{d-1}$ the fiber $(\nabla f)^{-1}(\nabla f(P))$ contains another point $Q\neq P$. Note that the general fiber is zero dimensional, otherwise the Hessian vanishes identically. We may assume that $Q$ is not in the base locus of $V$, otherwise $f$ is homaloidal. Restrict the polar map to the line $PQ$. Since a special member of $V$ cuts the line in two distinct points $P$, $Q$, by semicontinuity the general member of $V$ cuts the line in two distinct points. Note that the members of the system containing $P$ and another general point $P'\in PQ$ contains the three points $P$, $Q$ and $P'$, hence the whole line. It follows there  is a codimension two subspace $V_0\subset V$ such that the quadrics in $V_0$ contain the line $PQ$. Hence we get by restriction a map $F\colon PQ\to\mathbbm{P}((V/V_0)^\vee)=\mathbbm{P}^1$
defined by $R\mapsto \{H\subset V/V_0 \mid H(R)=0\}$. This is a morphism with quadratic entries and we get a pencil of quadrics on $\mathbbm{P}^1$ with no base points. By Lemma \ref{lem:pencil},
the pencil has two distinct points $M$, $N$,
where it ramifies. At $M$, by Proposition \ref{prop:contains_line}, there is one quadric singular at $M$. Hence $\Hess(f)$ vanishes at $M$. The same argument shows that $\Hess(f)$ vanishes at $N$. But this contradicts the assumption that  $\Hess(f)$ is the power of a nonzero linear form. 
This concludes the proof.
\end{proof}

\begin{remark}\label{rem:ST} The classification of homaloidal cubic surfaces has appeared in \cite[\S 5.1]{ST}, as application of results from \cite{DP, Huh}. The result is that the homaloidal cubic surfaces are the two persistent cubic surfaces found and the additional case of the union of a quadratic cone with a general hyperplane,
like $(x_0x_1+x_2^2)(x_0+x_3)$.
This last polynomial has Hessian equal to
$4(x_0x_1+x_2^2)(x_0+x_3)^2$, hence it is not persistent according to Theorem \ref{thm:main} (d).
\end{remark}

\subsection{Examples of Persistent Polynomials}
This section and the next one owe a lot to \S 4  of Ref. \cite{CRS}. Next Proposition is our main construction of examples of persistent tensors.

\begin{proposition}\label{prop:construct}
If $f\in\text{Sym}^n\mathbbm{C}^d$ is a general isobaric polynomial of weight $d-1$, with $n\geq 2$,
then $f$ is persistent. 
\end{proposition}
\begin{proof}
It is enough to combine Lemma \ref{lem:weight} with the implication $(\textnormal{a})$ $\Longrightarrow$ $(\textnormal{b})$ of Theorem \ref{thm:main}.
\end{proof}

Note that Theorem \ref{thm:small_dim} shows that
the construction in Proposition \ref{prop:construct} exhausts all the examples for $d\leq 4$.

The examples $\mathcal{L}(d,n)$ and $\mathcal{M}(d,n)$ in Ref.~\cite[(29)-(30)]{GL24} have indeed weight $d-1$.

We get examples of irreducible persistent polynomials on any $\mathbbm{P}^{d-1}$ with $d\geq 3$.

\begin{proposition}
Let $f$ be a irreducible persistent polynomial. Then the dual variety $V(f)^\vee$ is a hypersurface, equivalently the variety $V(f)$ is not dual defective.
\end{proposition}

\begin{proof}
By Segre formula (see \cite[Lemma 7.2.7]{Ru}) $V(f)^\vee$ is a hypersurface if and only if $f$ does not divide $\Hess(f)$. $f$ cannot divide $\Hess(f)$ by condition $(\textnormal{d})$ of Theorem \ref{thm:main}, by degree reasons and the assumption of irreducibility.
\end{proof}

\subsection{Prehomogeneous Spaces and Perazzo Example Revisited}\

\begin{definition}
Let $V$ be a complex vector space of dimension $d$ endowed with a $\mathtt{G}$-module structure, where $\mathtt{G}$ is a linear group, not necessarily reductive.  
The group $\mathrm{GL}(V)$ acts on polynomials $f$ on $V^\vee$ via $(g \cdot f)(x) = f(g^{-1} x), \, g \in \mathrm{GL}(V), \, x \in V^\vee$. A variety $\mathfrak{X} \subset \mathbbm{P}(V)$ that is $\mathtt{G}$-invariant has an induced $\mathtt{G}$-action on its homogeneous coordinate ring. If the $\mathtt{G}$-action has a dense orbit  $\mathfrak{X} \setminus V(f)$, the variety $\mathfrak{X}$ is called a \emph{prehomogeneous space}.
\end{definition}

An element $f$ of the homogeneous coordinate ring of $\mathfrak{X}$ is called a $\mathtt{G}$-semi-invariant if there exists a character 
$\chi \colon \mathtt{G} \to \mathbbm{C}^*$ such that $g \cdot f = \chi(g) f \text{ for all } g \in \mathtt{G}$. The set of all $\mathtt{G}$-semi-invariants forms a subring, called the ring of semi-invariants.  

If $\mathfrak{X}$ is prehomogeneous with dense $\mathtt{G}$-orbit $\mathfrak{X} \setminus V(f)$, and $f = f_1 \cdots f_k$ is its decomposition into irreducible factors, then the factors $f_i$ generate the ring of $\mathtt{G}$-semi-invariants on $X$.  
Moreover, the Hessian of a semi-invariant is again a semi-invariant. In particular, if $\mathfrak{X} = \mathbbm{P}^{d-1}$, it follows that
\begin{equation}
\Hess(f) = c \, f_1^{n_1} \cdots f_k^{n_k}\,, \quad \text{for some}~~c\in\mathbbm{C},\, n_i\in\mathbbm{N}\,.
\end{equation}

Let $\mathbbm{P}^2 = \mathrm{SL}(3)/\mathtt{P}$, where $\mathtt{P}$ is the parabolic subgroup of linear transformations that fix a point $x$, namely,
\begin{equation}\label{eq:parabolicP}
\mathtt{P} = \begin{pmatrix}
a_{00} & a_{01} & a_{02} \\
0 & a_{11} & a_{12} \\
0 & a_{21} & a_{22}
\end{pmatrix} \subset \mathrm{SL}(3)\,,
\end{equation}
if $x=\begin{pmatrix}1&0&0\end{pmatrix}^{\rm T}$ with ${\rm T}$ denoting transposition. We emphasize that $\mathtt{P}$ is not reductive. For instance, $\mathtt{P}$ acts on $\mathbbm{C}^3$ by left multiplication, and $\langle e_1 \rangle \subset \mathbbm{C}^3$ is a $\mathtt{P}$-submodule without any direct summand; that is, for any subspace $W$ such that $\langle e_1 \rangle \oplus W = \mathbbm{C}^3$, the subspace $W$ is not $\mathtt{P}$-invariant. In fact, examples such as the $\mathbbm{P}^4$ described in the next Proposition \ref{prop:P4prehom} fall outside the Sato–Kimura classification of prehomogeneous spaces, which applies only to reductive groups \cite{SK}. 
 
\begin{proposition}\label{prop:P4prehom}
$\mathtt{P}$ acts with a dense orbit on $\mathbbm{P}^4=\mathbbm{P}\big(H^0(I_x(2))\big)$.
The complement of the dense orbit is the variety of singular conics through $x$,
which is projectively equivalent to the cubic Perazzo hypersurface of Example \ref{exa:weight4}.
\end{proposition}
\begin{proof}
Every conic not passing through $x$ corresponds to a $3\times 3$ symmetric matrix 
$C = (c_{ij})_{i,j=0,1,2}$ with $c_{00} \neq 0$.  
The element
\begin{equation}
p_0 = \begin{pmatrix}
1 & -\frac{c_{01}}{c_{00}} & -\frac{c_{02}}{c_{00}} \\
0 & 1 & 0 \\
0 & 0 & 1
\end{pmatrix} \in \mathtt{P}\,,
\end{equation}
satisfies
\begin{equation}
p_0^{\rm T} C p_0 = \begin{pmatrix}
c_{00} & 0 & 0 \\
0 & * & * \\
0 & * & *
\end{pmatrix}.
\end{equation}

Assume now that $C$ is nonsingular. It is well known that one can reduce $C$ to a scalar multiple of the identity by acting with the $2\times 2$ SE-minor of elements in $\mathtt{P}$. This shows that the set of nonsingular conics through $x$ forms a dense orbit.  

Finally, consider the determinant
\begin{equation}\label{eq:perazzo}
\begin{vmatrix}
0 & c_{01} & c_{11} \\
c_{01} & c_{02} & c_{12} \\
c_{11} & c_{12} & c_{22}
\end{vmatrix}.
\end{equation}
Its variety is the complement of the dense orbit, it is left invariant by the action of $\mathtt{P}$ and it is projectively equivalent to the Perazzo hypersurface (see Example \ref{exa:weight4} and \cite[\S 7.4]{Ru} for generalizations).
\end{proof}

We now aim to determine all the orbits of the action of $\mathtt{P}$ on $\mathbbm{P}^4$ described in Proposition \ref{prop:P4prehom}, as well as those for the dual action.   The dual action can be interpreted as an action on hyperplanes, so the problem reduces to classifying the possible hyperplane sections of the Perazzo hypersurface. Although this is well known, we did not find an exposition of it in the literature in the context of group actions.

It is classically known that, up to the $\mathrm{SL}(4)$-action, there are two families of irreducible cubic surfaces in $\mathbbm{P}^3$ with infinitely many singular points, namely:
\begin{enumerate}
\item {the surface $V(x_0^2x_2+x_1^2x_3)$, with singular locus the line $V(x_0, x_1)$
and Hessian $x_0^2x_1^2$.}
\item  {the Chasles-Cayley surface
\begin{equation}\label{eq:chca}V\begin{vmatrix}0&x_0&x_1\\x_0&x_1&x_2\\x_1&x_2&x_3\end{vmatrix},\end{equation}
with singular locus the line $V(x_0, x_1)$ and Hessian $x_0^4$.}
\end{enumerate}

Both cases are obtained as hyperplane sections of the Perazzo hypersurface (for the Chasles–Cayley surface, note the equality $c_{02} = c_{11}$ in (\ref{eq:perazzo})).  
The remaining two orbits correspond to:
\begin{itemize}
\item a quadric surface with a tangent plane,
\item $V(x_0^2 x_1)$.
\end{itemize}

We conclude with the following:

\begin{proposition}
The (closures of the) $\mathtt{P}$-orbits on $\mathbbm{P}^4$ are
\[
\underbrace{\text{conic when (\ref{eq:chca}) has rank one}}_{\mathfrak{X}_1} \subset 
\underbrace{\mathbbm{P}^2 = V(c_{01},c_{11}) = \text{Sing}(\text{Perazzo})}_{\mathfrak{X}_2} \subset 
\underbrace{\text{Perazzo}}_{\mathfrak{X}_3} \subset \mathbbm{P}^4\,,
\]
while on the dual ${\mathbbm{P}^4}^\vee$ they are
\[
\underbrace{\mathbbm{P}^1 = \text{exceptional divisor}}_{\mathfrak{X}_2^\vee} \subset 
\underbrace{\text{surface cubic scroll}}_{\mathfrak{X}_3^\vee} \subset 
\underbrace{\text{quadric cone}}_{\mathfrak{X}_1^\vee} \subset {\mathbbm{P}^4}^\vee.
\]
It is interesting to note that the chain of varieties $\mathfrak{X}_i$ is \emph{not} inverted when passing to the dual $\mathfrak{X}_i^\vee$.
\end{proposition}

This technique can be extended to establish the following result (the Perazzo hypersurface corresponds to the case $m=3$).

\begin{proposition}\label{prop:van_hess}
Let $S_m$ be the $m\times m$ symmetric matrix with first entry $S_{0,0}=0$ and the remaining entries filled with $\binom{m+1}{2}-1$ indeterminates. For $m\geq 3$, the polynomial $\det(S_m)$ is concise, irreducible of degree $m$, and satisfies
\begin{equation}
\Hess\big(\det(S_m)\big) = 0\,.
\end{equation}
\end{proposition}
\begin{proof}
Let $\mathtt{P} \subset \mathrm{SL}(m)$ be the parabolic subgroup of linear transformations that fix 
$x = \begin{pmatrix} 1 & 0 & \cdots & 0 \end{pmatrix}^{\rm T}$. The same argument as in the proof of Proposition \ref{prop:P4prehom} shows that $\mathtt{P}$ acts with a dense orbit on $\mathbbm{P}^{\binom{m+1}{2}-2} = \mathbbm{P}\big(H^0(I_x(2))\big)$. The ring of semi-invariants is generated by $\det(S_m)$, which has degree $m$.  
Its Hessian is a semi-invariant which has  degree $\bigl(\binom{m+1}{2}-1\bigr)(m-2)$, which is not a multiple of $m$ for $m \geq 3$; indeed, it has remainder $2$ modulo $m$.
Hence the Hessian vanishes.
\end{proof}

\begin{remark}
In \cite[Example 4.3]{CNOS} it was computed a $5$-dimensional group which leaves invariant the Perazzo hypersurface, which had been interpreted as a Taylor variety. That group is a subgroup of the $6$-dimensional parabolic group $\mathtt{P}$, so that Propositions \ref{prop:P4prehom} and \ref{prop:van_hess} can be seen as a completion to \cite[Example 4.3]{CNOS}.
See also \cite[Remark 2.21]{COCDR}.
\end{remark}

\subsection{Some Persistent Polynomials are Semi-Invariants and Homaloidal}

In \cite{CRS}, a special class of examples, analogous to those in Proposition \ref{prop:construct}, was constructed using weighted determinants. These examples further exhibit the property of having a comparatively large isotropy group relative to a generic polynomial on $\mathbbm{C}^d$ of weight $d-1$.

As in Ref. \cite[\S 4.1.1]{CRS}, we define the generic sub-Hankel matrix ${\rm M}_{(d)}=\bigl(x_{\,i+j-d+1} \bigr)_{1\leq i,j \leq d-1}$ where we set $x_{k}=0 $ for $k\notin[0,d-1]$. Explicitly, one has
\begin{equation}
\renewcommand{\arraystretch}{1.5}
{\rm M}_{(d)}=
\left(\hspace{-1mm}\begin{NiceMatrix}[columns-width=1.7em]
0     &\Cdots &\Cdots & 0     & x_0   & x_1    \\
\Vdots&       &\Iddots& x_0   & x_1   & x_2    \\
\Vdots&\Iddots&\Iddots&\Iddots &\Iddots &\vdots  \\
0     & x_0   &\Iddots&\Iddots&\iddots&x_{d-3} \\
x_0   & x_1   &\Iddots&\iddots&x_{d-3}&x_{d-2} \\
x_1   & x_2   &\cdots &x_{d-3}&x_{d-2}&x_{d-1} \\
\end{NiceMatrix}\right).
\end{equation}

We then define the associated determinant
\begin{equation}
f_{(d)} = \det {\rm M}_{(d)}\,.
\end{equation}
The polynomial $f_{(d)}$ has both weight $d-1$ and degree $d-1$. Hence, by Proposition~\ref{prop:construct}, every $f_{(d)}$ with $d \geq 3$ is a persistent polynomial.  
Moreover, for $d \geq 3$ and $k \geq 0$, any reducible polynomial of the form $f_{(d)} x_0^k$ is also persistent by the same reasoning.

Note that $f_{(4)}$ corresponds precisely to the Chasles-Cayley surface in (\ref{eq:chca}). The polynomial $f_{(4)}$ admits a $3$-dimensional isotropy group given by
\begin{equation}\label{eq:P4}
\mathtt{P}_{(4)}=\left\{\begin{pmatrix}
1&a_1&a_2\\
0&a_0&2a_0a_1\\
0&0&a_0^2
\end{pmatrix} \Big|~a_0\neq 0\right\}.
\end{equation}
The group $\mathtt{P}_{(4)}$ is a subgroup of the parabolic group $\mathtt{P}$ in (\ref{eq:parabolicP}) and acts on $\mathbbm{P}^3$ by congruence, in the sense that
for any $p\in \mathtt{P}_{(4)}$, representing $(x_0,\ldots, x_3)$ by the entries of ${\rm M}_{(4)}$, the action is 
$p^{\rm T}{\rm M}_{(4)}p$. In other words, the restriction of the action of $\mathtt{P}$ on $\mathbbm{P}^4$ to the subgroup $\mathtt{P}_{(4)}$ leaves invariant the hyperplane corresponding to our $\mathbbm{P}^3$.
$\mathbbm{P}^3$ is a prehomogeneous space for this $\mathtt{P}_{(4)}$-action, with a dense orbit given by
$\mathbbm{P}^3\setminus V(x_0f_{(4)})$.
The ring of semi-invariants is generated by 
$x_0$ and $f_{(4)}$. We just observed that $x_0f_{(4)}$ is persistent
by Proposition \ref{prop:construct}. By \cite[Lemma 1]{Dol} the semi-invariant $x_0f_{(4)}$ is also homaloidal, indeed the extra assumption of regularity needed in \cite{Dol} is always satisfied in our cases, since for persistent polynomial the Hessian does not vanish by condition $(\textnormal{d})$ of Theorem \ref{thm:main}.
Indeed any power $f_{(4)}^kx_0^h$
is homaloidal for $k\geq 1$, $h\geq 0$, since by
the results by Dimca and Papadima in \cite{DP} the degree of the polar map $\nabla f$ is equal to the degree of $\nabla f_{\rm red}$ where $f_{\rm red}$ is reduced, with the same irreducible factors as $f$.

\begin{proposition}\label{prop:f4khess}
For all integers $k \geq 1$ and $h \geq 0$, one has
\begin{equation}
\Hess\big(f_{(4)}^k x_0^h\big)\,=\,\lambda \, f_{(4)}^{\,4(k-1)} x_0^{\,4(h+1)}\,, \quad \text{for some}~~\lambda \in \mathbbm{C}^*.
\end{equation}
\end{proposition}
\begin{proof}
Since $\mathtt{P}_{(4)}$ acts on $\mathbbm{P}^3$
with a dense orbit with complement $V(f_{(4)}x_0)$, we get that the semi-invariant ring is generated by $f_{(4)}$ and by $x_0$.
Hence $\Hess(f_{(4)}^kx_0^h)$ is a scalar multiple of $f_{(4)}^\alpha x_0^\beta$ for some $\alpha$, $\beta$ such that
$3\alpha+\beta=4(3k+h-2)$ by degree reasons.
It is enough to compute $\alpha$ and for this purpose it is enough to detect the appearance of
the variable $x_3$ in the Hessian matrix, since
it appears in the monomial $x_0^2x_3$ of $f_{(4)}$. Computing the second derivatives, the maximum exponents of $x_3$ in each entry are the following
\begin{equation}
\begin{pmatrix}
k&k-1&k-1&k-1\\
k-1&k-1&k-1&k-2\\
k-1&k-1&k-1&k-2\\
k-1&k-2&k-2&k-2\end{pmatrix}.
\end{equation}
In the diagonal there is the maximum appearance at exponent $4(k-1)$, hence $\alpha=4(k-1)$, as we wanted.
\end{proof}
The formula in Proposition \ref{prop:f4khess} reminds Proposition 2.7 in \cite{COCDR} and confirms, by condition $(\textnormal{d})$ of Theorem \ref{thm:main}, that $f_{(4)}^kx_0^h$ is persistent if and only if $k=1$ and $h\geq 0$.

The dual variety to the Chasles-Cayley cubic $V(f_{(4)})$ is isomorphic to $V(f_{(4)})$ itself. This is well known, see for example \cite[Proposition 1.7]{CRS}. An alternative proof
uses the fact that $f_{(4)}$ is an invariant for the action of $\mathtt{P}_{(4)}$ defined in (\ref{eq:P4}). In \cite{ST}, using results of \cite{DP, Huh}, it is proved that the only irreducible homaloidal cubic surface is the Chasles-Cayley surface, see Remark\ref{rem:ST}.

\begin{remark}
Examples suggest that the dual of a irreducible persistent polynomial of degree $d-1$ and weight $d-1$ in $\mathbbm{P}^{d-1}$ ($d\geq 3$)
is $\mathrm{SL}(d)$-equivalent to another irreducible persistent polynomial of degree $d-1$ and weight $d-1$ in $\mathbbm{P}^{d-1}$, so there is likely an involution in the $\mathrm{SL}(d)$-quotient. It would be interesting to check if this hypothesis is true.
\end{remark}

The polynomial $f_{(5)}$ has a $4$-dimensional isotropy group given by
\begin{equation}
\mathtt{P}_{(5)}=\left\{
\begin{pmatrix}
1&a_{1}&a_{2}&a_{3}\\
0&a_{0}&2\,a_{0}a_{1}&a_{0}a_{1}^{2}+2\,a_{0}a_{2}\\
0&0&a_{0}^{2}&3\,a_{0}^{2}a_{1}\\
0&0&0&a_{0}^{3}
\end{pmatrix} \Big|~a_0\neq 0\right\}.
\end{equation}
The action of $\mathtt{P}_{(5)}$ on $\mathbbm{P}^4$ is completely analogous to the action of $\mathtt{P}_{(4)}$ on $\mathbbm{P}^3$ described after (\ref{eq:P4}) and we do not repeat it. $\mathbbm{P}^4$ is a prehomogeneous space for this action with a regular invariant given by $x_0f_{(5)}$. By \cite[Lemma~1]{Dol}, this invariant is homaloidal. Moreover, by Proposition~\ref{prop:construct}, it is also persistent.  
It is interesting to compute
\begin{equation}
\Hess\big(f_{(5)}^k x_0^h\big)\,=\,\lambda\, f_{(5)}^{\,5(k-1)} x_0^{\,5(h+2)}\,, \quad \text{for some}~~\lambda\in\mathbbm{C}^*, \;\; \forall\, k \geq 1, \, h \geq 0\,.
\end{equation}
We omit the proof, as it is analogous to Proposition~\ref{prop:f4khess}.  
Once again, $f_{(5)}^k x_0^h$ is persistent if and only if $k=1$ and $h \geq 0$.


It seems reasonable to conjecture that symmetric persistent tensors are homaloidal; we have checked this property in all the examples we know. The fact that $f_{(d)}$ are homaloidal is proved in \cite[Theorem 4.4]{CRS}. The converse implication is not true by the example $f=x_0x_1x_2$, which is homaloidal but not persistent.
A weaker version of the question is the following. If the Hessian of a homogeneous polynomial is a power of a linear form, does it follow that the polynomial is homaloidal?  This question should be compared with the question posed by Ciliberto, Russo and Simis in \cite[Remark 3.5]{CRS}.

\subsection{Multiplicative Legendre Transform }

We follow closely Refs. \cite{Dol, EKP} for the preliminaries. Let $V$ be a complex vector space of dimension $d$.
Given a sufficiently regular homogeneous function $f$ on $V$,
we denote $d \ln f=\frac{\nabla f}{f}$. When $\Hess(\ln f)\neq 0$ there exists  a function $f_*$ on $V^\vee$ such that $f_*(d\ln f)=\frac{1}{f}$ (see Ref. \cite[Lemma 1]{Dol}).
$f_*$ is called the {\it multiplicative Legendre transform of $f$}.
\begin{theorem} \cite{Dol, EKP}
Let $f$ be a polynomial on $V$ with $\Hess(\ln f)\neq 0$. Then $f$ is homaloidal if and only if $f_*$ is rational.
\end{theorem}

When $f$ is homaloidal, $f_*$ gives the inverse of the polar map, in the sense that
\begin{equation}
\left(\frac{\nabla f}{f}\right)^{-1} = \frac{\nabla f_*}{f_*}\,.
\end{equation}

Note that $f_{(3)}$ is a smooth conic and its Legendre transform is projectively equivalent to itself.
Dolgachev computes for our persistent example $f=f_{(3)}=(x_0x_2+x_1^2)x_0$, 
$f_*=\frac{(4x_0x_2+x_1^2)^2}{x_2}$ (we keep the same sign as in \cite{Dol}, which of course does not influence the result).
Note the conic in the numerator of $f_*$ is the dual conic of $(x_0x_2+x_1^2)$, and all the unions of a smooth conic with a tangent line are projectively equivalent. This computation may be summarized saying that $\left(f_{3}\right)_*$
is projectively equivalent to $\frac{f_{(3)}^2}{x_0^3}$, a rational function but not a polynomial.

We compute now a similar pattern for the Legendre transform of our favorite persistent polynomials.

\begin{proposition}\label{prop:legendre}
The Legendre transform $\left(f_{(4)}\right)_*$ is projectively equivalent to $\frac{f_{(4)}^2}{x_0^3}$.

Moreover the Legendre transform $\left(f_{(4)}x_0\right)_*$ is projectively equivalent to $\frac{f_{(4)}^3}{x_0^5}$. In both cases we get a rational function, but not a polynomial.
\end{proposition}
\begin{proof}
The computations, on the lines of Example 5 in Ref. \cite{Dol}, are nontrivial. We took advantage of the software Macaulay2~\cite{GS}. The key commands are the following
\begin{verbatim}
    R=QQ[x_0..x_3,y_0..y_3]
    x1=matrix{{x_0..x_3}},y1=matrix{{y_0..y_3}}
    f=det matrix{{0,x_0,x_1},{x_0,x_1,x_2},{x_1,x_2,x_3}}
    J=ideal matrix {apply(4,i->(y_i*f-diff(x_i,f)))}
    J1=saturate(J,ideal(f))
\end{verbatim}

We get a system of four bilinear equations in $x,y$, equal to some scalars, which is nonsingular and can be solved. We started with $y_i$ as rational functions of $x_j$, we end with $x_j$ as rational functions of $y_i$. This computation looks similar to Ref. \cite[Proposition 4.5]{CRS}.
\end{proof}

It is likely that the pattern of Proposition \ref{prop:legendre} continues for the higher persistent polynomials $f_{(k)}$ and we leave it as an open problem.

A beautiful reference on the connection of homaloidal polynomials with prehomogeneous spaces is Ref. \cite{CS}. They analyze, as in \cite{EKP}, the cases when the Legendre transform is still a polynomial, expected to be very few.

\section*{Acknowledgments}
The second author is member of Italian GNSAGA and it is supported by the European Union - Next Generation EU, M4C1, CUP B53D23009560006, PRIN 2022- 2022ZRRL4C - Multilinear Algebraic Geometry, and by the HORIZON-MSCA-2023/Doctoral Network/Joint Doctorate call
TENORS, Grant agreement 101120296. The second author thanks Ciro Ciliberto and Francesco Russo for helpful discussions and for pointing \S 4 of Ref. \cite{CRS}
and Remark \ref{rem:bsegre}. 



\end{document}